\documentclass[10pt]{article}
\usepackage{mathrsfs}
\usepackage{amsfonts}
\usepackage{titlesec}
\usepackage{enumerate}
\usepackage[toc,page,title,titletoc,header]{appendix}
\usepackage{amsthm,amsmath,amssymb,anysize}
\usepackage[numbers,sort&compress]{natbib}
\theoremstyle{definition}
\newtheorem{theorem}{Theorem}
\newtheorem{lemma}{Lemma}[section]

\newtheorem{proposition}[lemma]{Proposition}

\newtheorem{example}[lemma]{Example}
\newcommand{\periodafter}[1]{#1.}
\titleformat{\subsection}[runin]
{\normalfont\bfseries}{\thesubsection}{0.5em}{\periodafter}
\marginsize{4.2cm}{4.2cm}{3.6cm}{3cm}%×óÓÒÉÏÏÂ
\numberwithin{equation}{section}
\renewcommand{\thelemma}{\arabic{section}.\arabic{lemma}}
\renewcommand{\thesection}{\arabic{section}.}
\renewcommand{\thesubsection}{\arabic{section}.\arabic{subsection}.}
\renewcommand{\theequation}{\arabic{section}.\arabic{equation}}

\newcounter{RomanNumber}

\newcommand{\MyRoman}[1]{\setcounter{RomanNumber}{#1}\Roman{RomanNumber}}
\title{\textsf{Restricted Kac modules of   Hamiltonian Lie superalgebras of odd type}}
\author{\textsc{Jixia Yuan$^{1,}$}\footnote{Supported by the NSF of  HLJ Provincial Education Department (12521158)}\;\;
\textsc{Wende Liu$^{2,}$}\footnote{Corresponding author. Email:
\texttt{wendeliu@ustc.edu.cn}.\;\;Supported by the NNSF
  of China (11171055)}\;  \\
  \textit{$^{1}$School of Mathematical Sciences},
  \textit{Heilongjiang University} \\
  \textit{Harbin 150080, China}\\
    \ \ \textit{$^{3}$School of Mathematical Sciences},
  \textit{Harbin Normal University} \\
  \textit{Harbin 150025, China}
  }
\date{ }
\begin{document}
\maketitle
\begin{quotation}
\noindent\textbf{Abstract} This paper aims  to describe  the restricted Kac modules of   restricted  Hamiltonian
Lie superalgebras of odd type over an algebraically closed field of characteristic $p>3$. In particular, a sufficient and necessary condition for the restricted Kac modules to be irreducible is given in terms of typical weights.
\\



  \end{quotation}

  \setcounter{section}{-1}
\section{Introduction}

\noindent
Over an algebraically closed field  of characteristic zero, the finite-dimensional or infinite-dimensional irreducible modules were  studied  for
  the finite-dimensional simple Cartan Lie superalgebras $W(n), S(n)$ and $ H(n)$ (cf. \cite{S} and references therein).
 Over a field of characteristic $p>0$, the analogs of these simple Cartan Lie superalgebras of characteristic zero are simple restricted Lie supealgebras with respect to the usual $p$-mapping ($p$-th associative power). Actually, in the characteristic $p>0$ case, there are more finite-dimensional  simple restricted Lie superalgebras, which are analogous to the finite-dimensional simple modular Lie algebras of Cartan type or infinite-dimensional simple Lie superalgebras of vector  fields over $\mathbb{C}$.  Among them are the first four series of
  finite-dimensional  simple graded restricted Lie superalgebras of Cartan type $W, S, H$ or $K$ (cf. \cite{Z}),   which are analogous to the finite-dimensional simple restricted Lie algebras of the  Cartan type $W, S, H$ or $K$ (cf. \cite{SF}), respectively. Representations of these four series of restricted Lie superalgebras were studied by Shu, Zhang and Yao (cf. \cite{Y1,Y2,YS,SZ,SZ1}). Additionally, there are four series of finite-dimensional graded simple restricted  Lie superalgebras of type $HO, KO, SHO$ or $SKO$ (cf. \cite{BL}), which are analogous to the infinite-dimensional simple  Lie superalgebras of vector fields over $\mathbb{C}$ (cf. \cite{k2}).  As far as we know, irreducible modules of the last four series of modular simple Lie superalgebras have not
been studied yet.
The aim of this paper is to make an attempt to describe the restricted simple modules  for the so-called  Hamiltonian Lie superalgebras of odd type over a field of prime characteristic.

\section{Preliminaries}

The ground field $\mathbb{F}$ is assumed to be algebraically closed  of characteristic $p>3$. Write
$\mathbb{Z}_2= \{\bar{0},\bar{1}\}$ for the additive group of two
elements, $\mathbb{Z}$ and $\mathbb{N}$   the sets of
integers and nonnegative integers, respectively.  For a vector superspace $V=V_{\bar{0}}\oplus
V_{\bar{1}}, $ we write $|x|=\theta$ for the
 {parity} of a $\mathbb{Z}_{2}$-homogeneous element $x\in V_{\theta}, $
$\theta\in \mathbb{Z}_{2}.$ The notation $|x|$ appearing in the text implies that $x$ is  a $\mathbb{Z}_2$-homogeneous element.

\subsection{Divided power superalgebras}
Fix  a pair of positive integers $m, n$  and write $\underline{r}=(r_1,\ldots,r_m\mid r_{m+1},\ldots,r_{m+n})$ for a $(m+n)$-tuple of non-negative integers. For an $m$-tuple $\underline{N}=(N_1,\ldots,N_{m})$, following \cite{Leb,Lei}, we write
 $\mathcal{O}(m, \underline{N}\mid n)$ for  the divided power superalgebra, which is a  supercommutative associative superalgebra  having a (homogeneous) basis:
 \begin{eqnarray*}
\{x^{(\underline{r})}\mid r_i<p^{N_i}\;\mbox{for}\; i\leq m \;\mbox{and}\; r_i=0\;\mbox{or}\;1 \;\mbox{for}\; i> m \}
 \end{eqnarray*}
 with parities
 $
 |x^{(\underline{r})}|=\left(\sum_{i>m}r_{i}\right)\bar{1}
$
 and multiplication relations:
 \begin{eqnarray*}
 x^{(\underline{r})} x^{(\underline{s})}=\Pi_{i=m+1}^{m+n}\min(1, 2-r_i-s_i)(-1)^{\Sigma_{m<i,j\leq m+n}r_js_i}\left(
\begin{array}{c}
\underline{r}+\underline{s} \\
 \underline{r}
\end{array}\right)x^{(\underline{r}+\underline{s})}.
 \end{eqnarray*}
As superalgebras,  $\mathcal{O}(m, \underline{N}\mid n)$ is isomorphic to the tensor product superalgebra of the divided power algebra $\mathcal{O}(m, \underline{N})$ with the trivial $\mathbb{Z}_2$-grading  and the exterior (super)algebra $\Lambda(n)$ with the usual $\mathbb{Z}_2$-grading:
$$
\mathcal{O}(m, \underline{N}\mid n)=\mathcal{O}(m, \underline{N})\otimes_{\mathbb{F}}\Lambda(n).
$$
\subsection{General vectorial Lie superalgebras}
Let $\epsilon_i$ be the $(m+n)$-tuple with 1 at the $i$-th slot and 0  elsewhere. For simplicity we write $x_i$ for $x^{(\epsilon_i)}$. Define the distinguished partial derivatives $\partial_i$ with parity $|\partial_i|=|x_i|$ by letting
$$
\partial_i\left(x^{(k\epsilon_j)}\right)=\delta_{ij}x^{((k-1)\epsilon_j)}\;\mbox{for  $k<p^{N_{j}}$}.
$$

The Lie superalgebra of $\mathcal{O}(m, \underline{N}\mid n)$ of all superderivations contains an important subalgebra, called the general vectorial Lie superalgebra of distinguished superderivations (a.k.a. the Lie superlagebra of Witt type), denoted by $\mathfrak{vect}(m, \underline{N}\mid n)$ (a.k.a.
$W(m, \underline{N}\mid n)$), having an $\mathbb{F}$-basis (see \cite{Leb,Lei} for more details)
$$
\left\{x^{(\underline{r})}\partial_k\mid r_i<p^{N_i}\;\mbox{for $i\leq m$; $1\leq k\leq m+n$}\right\}.
$$
Generally speaking, $\mathfrak{vect}(m, \underline{N}\mid n)$ contains various finite-dimensional simple Lie superalgebras, which are analogous to  finite-dimensional simple modular Lie algebras or infinite-dimensional simple Lie superalgebras of vector fields over $\mathbb{C}$, as mentioned in the introduction.

\subsection{Hamiltonian superalgebras of odd type}
From now on, suppose $m=n$. As in \cite{Leb,Lei}, write
$$
\mathrm{De}_{f}=\sum_{i=1}^{2n}(-1)^{|\partial_{i}||f|}\partial_{i}(f)\partial_{i^{'}},
$$
where
\[ i'=\left\{
 \begin{array}{ll}
i+n,&\mbox{if  $i\leq n,$ }\\
i-n,&\mbox{if $i>n.$}
\end {array}
\right.
\]
Note that
$$
|\mathrm{De}_{f}|=|f|+1
$$
and
\begin{equation*}\label{hee1.1}
   [\mathrm{De}_{f},\mathrm{De}_{g}]=\mathrm{De}_{\{f,g\}_{B}}\;\mbox{for}\;
   f,g\in\mathcal{O}(n,\underline{N}\mid n),
\end{equation*}
where $\{\cdot,\cdot\}_{B}$ is the Buttion bracket given by
$$
\{f,g\}_{B}=\mathrm{De}_{f}(g)=\sum_{i=1}^{2n}(-1)^{|\partial_{i}||f|}\partial_{i}(f)\partial_{i^{'}}(g).
$$
Then
\[\frak{le}\left(n, \underline{N}\mid n\right)=\{\mathrm{De}_{f}\mid f\in\mathcal{O}(n, \underline{N}\mid n)\}\]
is a finite-dimensional simple  Lie superalgebra. We call it the
Hamiltonian superalgebra of odd type. This Lie superalgebra  is called  the odd Hamiltonian superalgebra and denoted by $HO(n,n; \underline{N})$  in \cite{LZ}, which is analogous to the infinite-dimensional Lie superalgebra $HO(n,n)$ of vector fields  over $\mathbb{C}$ in   \cite{k2}. In the present paper, we adopt the notation in \cite{Leb,Lei}.

\subsection{Restricted Lie superalgebras and restricted modules}
Recall that a Lie superalgebra $\mathfrak{g}=\mathfrak{g}_{\bar{0}}\oplus \mathfrak{g}_{\bar{1}}$   is
 restricted if $\mathfrak{g}_{\bar{0}}$ as Lie algebra is  restricted   and
 $\mathfrak{g}_{\bar{1}}$ as $\mathfrak{g}_{\bar{0}}$-module is  restricted.
 For a restricted Lie superalgebra $\mathfrak{g}$,  the $p$-mapping of $\mathfrak{g}_{\bar{0}}$ is also called
the $p$-mapping of the Lie superalgebra $\mathfrak{g}$.

Let $(\mathfrak{g}, [p])$ be a restricted Lie superalgebra. A $\mathfrak{g}$-module $M$ is called restricted if
$$x^{p}\cdot m=x^{[p]}\cdot m \; \mbox{for all \;$x\in \mathfrak{g}_{\bar{0}}, m\in M$}.$$

  By definition, the restricted
enveloping algebra of $\mathfrak{g}$ is  $\mathbf{u}(\mathfrak{g})=\mathrm{U}(\mathfrak{g})/I$, where $I$ is the $\mathbb{Z}_{2}$-graded two-sided
ideal of enveloping algebra $\mathrm{U}(\mathfrak{g})$ generated by elements $\left\{x^{p}-x^{[p]} \mid x\in \mathfrak{g}_{\bar{0}}\right\}.$
Note that  $\mathbf{u}(\mathfrak{g})$ has  a natural structure of a $\mathbb{Z}$-graded superalgebra. Suppose  $(e_1,\ldots, e_m\mid f_1,\ldots, f_n)$ is a $\mathbb{Z}_{2}$-homogeneous basis of $\frak{g}$.
  Then $\mathfrak{u}(\mathfrak{g})$ has the following $\mathbb{F}$-basis:
\begin{eqnarray*}
\left\{f_{1}^{b_{1}}\cdots f_{n}^{b_{n}}e_{1}^{a_{1}}\cdots e_{m}^{a_{m}}\mid 0\leq a_{i}\leq p-1, b_{j}=0\; or \;1\right\}.
\end{eqnarray*}
Note that the $\mathbf{u}(\frak{g})$-modules are precisely the restricted   $\mathfrak{g}$-modules.

A standard fact is that
$\frak{le}(n, \underline{N}\mid n)$ is  a restricted Lie superalgebra if and  only if
$$\underline{N}=\underline{1}:=(1,\ldots,1).$$ Note that the unique  $p$-mapping of $\frak{le}(n, \underline{1}\mid n)$ is the usual (associative) $p$-power.

This paper aims to study the finite-dimensional  irreducible restricted modules of $\frak{le}(n, \underline{1}\mid n)$. From now on we abbreviate $\mathcal{O}(n, \underline{1}\mid n)$ to $\mathcal{O}(n)$, and $\frak{le}(n, \underline{1}\mid n)$ to $\frak{le}(n)$.

\section{Kac  modules and  root reflections}
For a $\mathbb{Z}$-graded  vector space   $V=\oplus_{i\in \mathbb{Z}}V_{[i]}$, we write $\mathrm{deg}v=i$ for the
$\mathbb{Z}$-degree of a $\mathbb{Z}$-homogeneous element $v\in V_{[i]}.$
Put
$$\overline{\frak{le}}(n)=\frak{le}(n)\oplus \mathbb{F}\sum_{ i=1}^{2n}x_{i}\partial_{i}.$$
By letting $\mathrm{deg}x_{i}=1=-\mathrm{deg}\partial_{i},$    $\mathfrak{le}(n)$ and $\overline{\frak{le}}(n)$   become     $\mathbb{Z}$-graded  superalgebras.
 %$$\mathcal{O}(n)=\oplus_{i=0}^{np}\mathcal{O}(n)_{[i]},$$
%$$\overline{\frak{le}}(n)=\oplus_{i=-1}^{np-2}\overline{\frak{le}}(n)_{[i]}.$$

\subsection{Triangular decompositions}
Let $\bar{\mathfrak{h}}=\mathfrak{h}\oplus \mathbb{F}\sum_{ i=1}^{2n}x_{i}\partial_{i},$ where
 $$\mathfrak{h}=\mathrm{Span}_{\mathbb{F}}\{\mathrm{De}_{x_{i}x_{i'}}\mid  i \leq n\}.$$ Then $\bar{\mathfrak{h}}$ is a Cartan subalgebra of  $\overline{\frak{le}}(n)_{[0]}$ and
$$\overline{\frak{le}}(n)=\oplus_{\alpha\in \bar{\mathfrak{h}}^{*}}\overline{\frak{le}}(n)_{\alpha},$$
where
$$\overline{\frak{le}}(n)_{\alpha}=\mathrm{Span}_{\mathbb{F}}\left\{x\in \overline{\frak{le}}(n)\mid [h, x]=\alpha(h)x \;\mbox{for} \; h\in \bar{\mathfrak{h}}\right\}.$$
We denote the  basis of  $\bar{\mathfrak{h}}^{*},$
$$\varepsilon_{i}=\left(\mathrm{De}_{x_{i}x_{i'}}\right)^{*},\; \delta=\left(\sum_{j=1}^{2n}x_{j}\partial_{j}\right)^{*} \; \mbox{for}\; i \leq n.$$
We still write $\varepsilon_{i}$ for $\varepsilon_{i}|_{\mathfrak{h}}$, if   no confusion occurs.
Clearly,
\[ \mathrm{De}_{x_{i}}\in\left\{
 \begin{array}{ll}
\overline{\frak{le}}(n)_{-\varepsilon_{i}-\delta}& \mbox{ if}\; i \leq n \\
\overline{\frak{le}}(n)_{\varepsilon_{i'}-\delta}& ~\mbox{if}\; i>n.
\end {array}
\right.
\]
Note that $\overline{\frak{le}}(n)_{[0]}$ has a standard triangular decomposition
$$\overline{\frak{le}}(n)_{[0]}=\mathfrak{n}_{[0]}^{-}\oplus \bar{\mathfrak{h}}\oplus\mathfrak{n}_{[0]}^{+},$$
where
$$\mathfrak{n}_{[0]}^{-}=\mathrm{Span}_{\mathbb{F}}\{\mathrm{De}_{x_{i}x_{n+j}}\mid n\geq i>j\}+\mathrm{Span}_{\mathbb{F}}\{\mathrm{De}_{x_{k}x_{l}}\mid   k, l>n\},$$
$$\mathfrak{n}_{[0]}^{+}=\mathrm{Span}_{\mathbb{F}}\{\mathrm{De}_{x_{i}x_{n+j}}\mid   i <j\leq n \}+\mathrm{Span}_{\mathbb{F}}\{\mathrm{De}_{x_{k}x_{l}}\mid   k, l\leq n\}.$$
Then $\overline{\frak{le}}(n)$ has a standard triangular decomposition
$$\overline{\frak{le}}(n)=\mathfrak{n}^{-}_{0}\oplus \bar{\mathfrak{h}}\oplus\mathfrak{n}^{+}_{0},$$
where
$$\mathfrak{n}^{-}_{0}=\mathfrak{n}^{-}_{[0]}\oplus \overline{\frak{le}}(n)_{[-1]},\;\mathfrak{n}^{+}_{0}=\mathfrak{n}^{+}_{[0]}\oplus_{i>0}\overline{\frak{le}}(n)_{[i]}.$$

\subsection{Root reflections}
We now define a sequence of root reflections  in the order: $\gamma_{-\varepsilon_{1}-\delta},\ldots,\gamma_{-\varepsilon_{n}-\delta}, \gamma_{\varepsilon_{n}-\delta},\ldots,$ $\gamma_{\varepsilon_{1}-\delta}.$
Firstly we define
$$\mathfrak{n}^{+}_{1}=\gamma_{-\varepsilon_{1}-\delta}(\mathfrak{n}^{+}_{0})$$
to be obtained by removing the subspace
$$W_{1}=\mathrm{Span}_{\mathbb{F}}\left\{\mathrm{De}_{x^{(\epsilon_{n+1}+\epsilon_{i}+\epsilon_{n+j})}}, \mathrm{De}_{x^{(\epsilon_{n+1}+\epsilon_{n+i}+\epsilon_{n+j})}}\mid n\geq i\geq j \right\}$$
from $\mathfrak{n}^{+}_{0}$ and adding $\mathbb{F}\mathrm{De}_{x_{1}}$.
Put
$$\mathfrak{n}^{-}_{1}=W_{1}\oplus\mathfrak{n}^{-}_{[0]}\oplus \sum_{i=2}^{2n}\mathbb{F}\mathrm{De}_{x_{i}}.$$
Then   $\overline{\frak{le}}(n)=\mathfrak{n}^{-}_{1}\oplus \bar{\mathfrak{h}} \oplus \mathfrak{n}^{+}_{1}$ is a new triangular decomposition.
Suppose we have defined
$$\mathfrak{n}^{+}_{k-1}=\gamma_{-\varepsilon_{k-1}-\delta}\cdots
\gamma_{-\varepsilon_{1}-\delta}(\mathfrak{n}^{+}_{0}) \;\mbox{for}\; 2\leq k\leq n.$$
Then we define
$$\mathfrak{n}^{+}_{k}=\gamma_{-\varepsilon_{k}-\delta}(\mathfrak{n}^{+}_{k-1})$$
to be obtained by removing the subspace $W_{k}$ spanned by all $\mathrm{De}_{x^{(\underline{r}+\epsilon_{n+i}+\epsilon_{n+j})}} $ and $ \mathrm{De}_{x^{(\underline{r}+\epsilon_{i}+\epsilon_{n+j})}}$ with
\begin{eqnarray*}
 r_{1}=\cdots=r_{n}=r_{n+k+1}=\cdots =r_{2n}=0, r_{n+k}= 1, n\geq i\geq j
\end{eqnarray*}
from $\mathfrak{n}^{+}_{k-1}$ and adding $\mathbb{F}\mathrm{De}_{x_{k}} $.
Put
$$\mathfrak{n}^{-}_{k}=\sum_{i=1}^{k}W_{i}\oplus\mathfrak{n}^{-}_{[0]}\oplus \sum_{i=k+1}^{2n}\mathbb{F}\mathrm{De}_{x_{i}}.$$
Then   $\overline{\frak{le}}(n)=\mathfrak{n}^{-}_{k}\oplus \bar{\mathfrak{h}} \oplus \mathfrak{n}^{+}_{k}$ is a new triangular decomposition.
Next we define
$$\mathfrak{n}^{+}_{n+1}=\gamma_{\varepsilon_{n}-\delta}(\mathfrak{n}^{+}_{n})$$
to be obtained by removing root the space $W_{n+1}$ spanned by
\begin{eqnarray*}
 \{ \mathrm{De}_{x^{(\underline{r}+\epsilon_{n+i}+\epsilon_{n+j})}}, \mathrm{De}_{x^{(\underline{r}+\epsilon_{i}+\epsilon_{n+j})}} \mid
  r_{1}=\cdots=r_{n-1}=0, r_{n}\geq1, n\geq i\geq j \}
\end{eqnarray*}
from $\mathfrak{n}^{+}_{n}$ and adding $\mathbb{F}\mathrm{De}_{x_{2n}}$.
Put
$$\mathfrak{n}^{-}_{n+1}=\sum_{i=1}^{n+1}W_{i}\oplus\mathfrak{n}^{-}_{[0]}\oplus \sum_{i=n+1}^{2n-1}\mathbb{F}\mathrm{De}_{x_{i}}.$$
Then   $\overline{\frak{le}}(n)=\mathfrak{n}^{-}_{n+1}\oplus \bar{\mathfrak{h}} \oplus \mathfrak{n}^{+}_{n+1}$ is a new triangular decomposition.
Suppose we have defined
$$\mathfrak{n}^{+}_{n+k}=\gamma_{\varepsilon_{n-(k-1)}-\delta}\cdots\gamma_{\varepsilon_{n}-\delta}
(\mathfrak{n}^{+}_{n}) \;\mbox{for}\; 2\leq k\leq n.$$
Finally  we define
$$\mathfrak{n}^{+}_{n+k+1}=\gamma_{\varepsilon_{n-k}-\delta}(\mathfrak{n}^{+}_{n+k})$$
to be obtained by removing the space $W_{n+k+1}$ spanned by all $ \mathrm{De}_{x^{(\underline{r}+\epsilon_{n+i}+\epsilon_{n+j})}}$ and $\mathrm{De}_{x^{(\underline{r}+\epsilon_{i}+\epsilon_{n+j})}}$ with
\begin{eqnarray*}
 r_{1}=\cdots=r_{n-k-1}=0, r_{n-k}\geq 1,\cdots,r_{n}\geq1, n\geq i\geq j
\end{eqnarray*}
from $\mathfrak{n}^{+}_{n+k}$ and adding $\mathbb{F} \mathrm{De}_{x_{2n-k}}$.
Put
$$\mathfrak{n}^{-}_{n+k+1}=\sum_{i=1}^{n+k+1}W_{i}\oplus\mathfrak{n}^{-}_{[0]}\oplus \sum_{i=k+1}^{2n-k-1}\mathbb{F}\mathrm{De}_{x_{i}}.$$
Then   $\overline{\frak{le}}(n)=\mathfrak{n}^{-}_{n+k+1}\oplus \bar{\mathfrak{h}} \oplus \mathfrak{n}^{+}_{n+k+1}$ is a new triangular decomposition.
Note that
$$\mathfrak{n}^{+}_{2n}=\mathfrak{n}^{+}_{[0]}\oplus\overline{\frak{le}}(n)_{[-1]}, \; \mathfrak{n}^{-}_{2n}=\mathfrak{n}^{-}_{[0]}\oplus_{i>0}\overline{\frak{le}}(n)_{[i]}.$$

\subsection{Restricted Vermas module and  Kac modules}

Write $\mathfrak{g}$ for $\frak{le}(n)$ or $\overline{\frak{le}}(n).$ Suppose $\mathfrak{g}$ (resp. $\mathfrak{g}_{[0]}$) has  a triangular decomposition
$$\mathfrak{g}=N^{-}\oplus H_{\mathfrak{g}}\oplus N^{+}\;(\mbox{resp.}\;  \mathfrak{g}_{[0]}=N^{-}_{[0]}\oplus H_{\mathfrak{g}}\oplus N^{+}_{[0]}),$$  where $H_{\mathfrak{g}}=\mathfrak{h}\oplus \delta_{\mathfrak{g},\overline{\frak{le}}(n)}\mathbb{F}\sum_{ i=1}^{2n}x_{i}\partial_{i}$. Let $V=V_{\bar{0}}\oplus V_{\bar{1}}$ be a $\mathfrak{g}$-module (resp. $\mathfrak{g}_{[0]}$-module). Suppose for $\lambda\in H_{\mathfrak{g}}^{*}$ there is a nonzero vector $v_{\lambda}\in V_{\bar{0}}\cup V_{\bar{1}}$
such that
\begin{equation}\label{liukks}
x\cdot v_{\lambda}=0,\; h \cdot v_{\lambda}=\lambda(h)v_{\lambda},\;\mbox{for\;} x\in N^{+} \;(\mbox{resp.}\;x\in N^{+}_{[0]}),\; h\in H_{\mathfrak{g}}.
\end{equation}
Then  $v_{\lambda}$ is called a highest weight  vector of $V$ with respect to $B=H_{\mathfrak{g}}\oplus N^{+}$  (resp. $B_{[0]}=H_{\mathfrak{g}}\oplus N^{+}_{0}$).
If $V$ is  a restricted $\mathfrak{g}$-module (resp. $\mathfrak{g}_{[0]}$-module), then $\lambda\in \Lambda_{\mathfrak{g}}$, where
\begin{eqnarray*}
&&\Lambda_{\mathfrak{g}}=\mathrm{Span}_{\mathbb{F}_{p}}\{\varepsilon_{1},\ldots,\varepsilon_{n},\delta_{\mathfrak{g},\overline{\frak{le}}(n)}\delta\}.
\end{eqnarray*}
 Conversely, for any given  $\lambda\in \Lambda_{\mathfrak{g}}$, there is a one-dimensional restricted  $B$-module (resp. $B_{[0]}$-module) $\mathbb{F}v_{\lambda}$ satisfying (\ref{liukks}). Write $L_{\mathfrak{g}}^{B}(\lambda)$  (resp. $L_{\mathfrak{g}_{[0]}}^{B_{[0]}}(\lambda)$) for the unique irreducible $\mathbb{Z}_{2}$-graded quotient of restricted Verma module $\mathbf{u}(\mathfrak{g})\otimes_{\mathbf{u}(B)}\mathbb{F}v_{\lambda}$ (resp. $\mathbf{u}(\mathfrak{g}_{[0]})\otimes_{\mathbf{u}(B_{[0]})}\mathbb{F}v_{\lambda}$) of $\mathfrak{g}$  (resp. $\mathfrak{g}_{[0]}$) with respect to $B$   (resp.  $B_{[0]}$).
Recall that, for $\lambda\in\Lambda_{\mathfrak{g}}$,
$$
I_{\mathfrak{g}}\left(\lambda\right)=\mathbf{u}(\mathfrak{g})\otimes_{\mathbf{u}(\oplus_{i\geq 0}\mathfrak{g}_{[i]})}L_{\mathfrak{g}_{[0]}}^{ \mathfrak{n}_{[0]}^{+}\oplus H_{\mathfrak{g}}}\left(\lambda \right)
$$  is called a restricted Kac module of $\mathfrak{g}$.
We know that $I_{\mathfrak{g}}\left(\lambda\right)$ has a unique irreducible quotient module $L_{\mathfrak{g}}^{\mathfrak{n}_{0}^{+}\oplus H_{\mathfrak{g}}}\left(\lambda\right)$  and $$\left\{L_{\mathfrak{g}}^{\mathfrak{n}_{0}^{+}\oplus H_{\mathfrak{g}}}(\lambda)\mid \lambda\in \Lambda_{\mathfrak{g}}\right\}$$ consists of  all the irreducible restricted $\mathfrak{g}$-modules.
 For   $\lambda\in\Lambda_{\mathfrak{le}(n)},$  it is easily seen that $I_{\frak{le}(n)}(\lambda)$ is
   irreducible as $\frak{le}(n)$-module if and only if $I_{\overline{\frak{le}}(n)}(\lambda)$ is irreducible as $\overline{\frak{le}}(n)$-module.
In the below, we
 shall  study the irreducibility of $I_{\overline{\frak{le}}(n)}(\lambda)$, where $\lambda\in\Lambda_{\overline{\mathfrak{le}}(n)}$.

\subsection{Irreducible quotients of restricted Verma modules} For $  i \leq n,$ put
$$\mathfrak{h}_{i}=\mathrm{Span}_{\mathbb{F}}\{\mathrm{De}_{x_{j}x_{j'}}\mid   j\leq n, j\neq i \}.$$
By the definition of $\mathfrak{h}_{i}$, one can easily get
$$\mathfrak{h}_{i}=\{h\in \bar{\mathfrak{h}}\mid \varepsilon_{i}(h)=\delta(h)=0\}.$$
 For $0 \leq i \leq 2n,$ put
$\mathfrak{b}_{i}=\mathfrak{n}^{+}_{i}\oplus\bar{\mathfrak{h}}.$
Then we have the following proposition.
\begin{proposition}\label{pl1}
Let $ i\leq n,$  and  $\lambda\in \Lambda_{\overline{\mathfrak{le}}(n)}.$

(1) If $\lambda(\mathfrak{h}_{i})=0,$ then
$$
L_{\overline{\frak{le}}(n)}^{\mathfrak{b}_{i-1}}(\lambda)\cong L_{\overline{\frak{le}}(n)}^{\mathfrak{b}_{i}}(\lambda).
$$

(2) If $\lambda(\mathfrak{h}_{i})\neq0,$ then
$$
L_{\overline{\frak{le}}(n)}^{\mathfrak{b}_{i-1}}(\lambda)\cong L_{\overline{\frak{le}}(n)}^{\mathfrak{b}_{i}}(\lambda-\varepsilon_{i}-\delta).$$

(3) If $\lambda(\mathfrak{h}_{n-i+1})\neq0,$ then
$$L_{\overline{\frak{le}}(n)}^{\mathfrak{b}_{n+i-1}}(\lambda)\cong L_{\overline{\frak{le}}(n)}^{\mathfrak{b}_{n+i}}(\lambda+(p-1)(\varepsilon_{n-i+1}-\delta)).
$$

(4) If $\lambda=a\delta,$ where $a\in \mathbb{F}_{p},$ then
$$
L_{\overline{\frak{le}}(n)}^{\mathfrak{b}_{n+i-1}}(\lambda)\cong L_{\overline{\frak{le}}(n)}^{\mathfrak{b}_{n+i}}(\lambda).
$$

(5) If $\lambda=a\varepsilon_{n-i+1}+b\delta,$ where $ a, b\in \mathbb{F}_{p},$ $a\neq 0, 1,$ then
$$L_{\overline{\frak{le}}(n)}^{\mathfrak{b}_{n+i-1}}(\lambda)\cong L_{\overline{\frak{le}}(n)}^{\mathfrak{b}_{n+i}}(\lambda+(p-1)(\varepsilon_{n-i+1}-\delta)).
$$

(6) If $\lambda=\varepsilon_{n-i+1}+b\delta,$ where $b\in \mathbb{F}_{p},$ then
$$L_{\overline{\frak{le}}(n)}^{\mathfrak{b}_{n+i-1}}(\lambda)\cong L_{\overline{\frak{le}}(n)}^{\mathfrak{b}_{n+i}}(\lambda+(p-2)(\varepsilon_{n-i+1}-\delta)).
$$

\end{proposition}
\begin{proof}
(1)
Let $0\neq v$  be a highest weight vector of  $L_{\overline{\frak{le}}(n)}^{\mathfrak{b}_{i-1}}(\lambda)$ with respect to $\mathfrak{b}_{i-1}$. We claim that $v$ is also a highest weight vector of  $L_{\overline{\frak{le}}(n)}^{\mathfrak{b}_{i-1}}(\lambda)$ with respect to $\mathfrak{b}_{i}$. By the definition of $\gamma_{-\varepsilon_{i}-\delta}$,  we have to check that $\mathrm{De}_{x_{i}}\cdot v=0$. Suppose  $\mathrm{De}_{x_{i}}\cdot v\neq0$. One can easily get
$$X\cdot (\mathrm{De}_{x_{i}}\cdot v)=0 \;\mbox{for}\; X\in \mathfrak{b}_{i-1}.$$
Therefore, $\mathrm{De}_{x_{i}}\cdot v$  is a highest weight vector of  $L_{\overline{\frak{le}}(n)}^{\mathfrak{b}_{i-1}}(\lambda)$  with respect to $\mathfrak{b}_{i-1}$, which contradicts the uniqueness of highest weight vector with respect to $\mathfrak{b}_{i-1}$ (up to proportionality).

(2)
Let $0\neq v$  be a highest weight vector of  $L_{\overline{\frak{le}}(n)}^{\mathfrak{b}_{i-1}}(\lambda)$ with respect to $\mathfrak{b}_{i-1}$. In this case there exists $j\neq i,$ where $ j \leq n$,
such that $\lambda\left(\mathrm{De}_{x_{j}x_{j'}}\right)\neq0.$
Since
\begin{eqnarray*}
\mathrm{De}_{x^{(\epsilon_{i'}+\epsilon_{j}+\epsilon_{j'})}}\cdot \mathrm{De}_{x_{i}}\cdot v&=& -\mathrm{De}_{x_{i}}\cdot \mathrm{De}_{x^{(\epsilon_{i'}+\epsilon_{j}+\epsilon_{j'})}} \cdot v\\
&&+\lambda\left(\left[\mathrm{De}_{x_{i}},\mathrm{De}_{x^{(\epsilon_{i'}+\epsilon_{j}+\epsilon_{j'})}} \right]\right)v\\
&=&\lambda\left(\mathrm{De}_{x_{j}x_{j'}}\right)v\neq0,
\end{eqnarray*}
we have $\mathrm{De}_{x_{i}}\cdot v\neq0.$ Then   $\mathrm{De}_{x_{i}}\cdot v$ is a highest weight vector of $L_{\overline{\frak{le}}(n)}^{\mathfrak{b}_{i-1}}(\lambda)$ with respect to $\mathfrak{b}_{i}.$

(3)  In this case there exists  $ j\neq n-i+1$, where $j \leq n$,
such that $\lambda\left(\mathrm{De}_{x_{j}x_{j'}}\right)\neq0.$ Let $0\neq v$  be a highest weight vector of  $L_{\overline{\frak{le}}(n)}^{\mathfrak{b}_{n+i-1}}(\lambda)$ with respect to $\mathfrak{b}_{n+i-1}.$
We claim that
$$ \mathrm{De}_{x_{(n-i+1)'}}^{p-1}\cdot v\neq0.$$ Suppose not. Then by applying $\mathrm{De}_{x^{(\epsilon_{n-i+1}+\epsilon_{j}+\epsilon_{j'})}},$ we get
\begin{eqnarray*}
0&=&\mathrm{De}_{x^{(\epsilon_{(n-i+1)'}+\epsilon_{j}+\epsilon_{j'})}}\cdot \mathrm{De}_{x_{(n-i+1)'}}^{p-1}\cdot v\\
&=& -\mathrm{De}_{x_{(n-i+1)'}}\cdot \mathrm{De}_{x^{(\epsilon_{(n-i+1)'}+\epsilon_{j}+\epsilon_{j'})}}\cdot \mathrm{De}_{x_{(n-i+1)'}}^{p-2}\cdot v\\
&&+\lambda\left(\left[\mathrm{De}_{x_{(n-i+1)'}},\mathrm{De}_{x^{(\epsilon_{(n-i+1)'}+\epsilon_{j}+\epsilon_{j'})}}\right]\right)
\mathrm{De}_{x_{(n-i+1)'}}^{p-2}\cdot v\\
&=&(1-p)\lambda\left(\mathrm{De}_{x_{j}x_{j'}}\right)\mathrm{De}_{x_{(n-i+1)'}}^{p-2}\cdot v,
\end{eqnarray*}
hence $\mathrm{De}_{x_{(n-i+1)'}}^{p-2}\cdot v=0.$ By repeated applications of $\mathrm{De}_{x^{(\epsilon_{(n-i+1)'}+\epsilon_{j}+\epsilon_{j'})}},$ we can get $v=0,$ a contradiction.
A direct verification  shows that  $\mathrm{De}_{x_{(n-i+1)'}}^{p-1}\cdot v$ is a highest weight vector of $L_{\overline{\frak{le}}(n)}^{\mathfrak{b}_{n+i-1}}(\lambda)$ with respect to $\mathfrak{b}_{n+i}.$

(4) The proof is similar to the one of (1).

(5) Let $0\neq v$  be a highest weight vector of  $L_{\overline{\frak{le}}(n)}^{\mathfrak{b}_{n+i-1}}(\lambda)$ with respect to $\mathfrak{b}_{n+i-1}.$
Using the fact that
\begin{eqnarray}\label{lye1}
\mathrm{De}_{x^{((p-1)\epsilon_{n-i+1}+\epsilon_{(n-i+1)'})}}\cdot \mathrm{De}_{x_{(n-i+1)'}}^{(p-2)}\cdot v=\lambda\left(\mathrm{De}_{x_{n-i+1}x_{(n-i+1)'}}\right)v=av\neq 0,
\end{eqnarray}
we get
$$ \mathrm{De}_{x_{(n-i+1)'}}^{(p-2)}\cdot v\neq 0.$$
Using the fact that
\begin{eqnarray*}
&&\mathrm{De}_{x^{(2\epsilon_{n-i+1}+\epsilon_{(n-i+1)'})}}\cdot \mathrm{De}_{x_{(n-i+1)'}}^{(p-1)}\cdot v\\
&=&\left((p-1)\lambda+\frac{1}{2}(p-1)(p-2)(\varepsilon_{n-i+1}-\delta)\right)\left(\mathrm{De}_{x_{n-i+1}x_{(n-i+1)'}}\right)
\mathrm{De}_{x_{(n-i+1)'}}^{p-2}v\\
&=&(1-a)\mathrm{De}_{x_{(n-i+1)'}}^{p-2}v\neq 0,
\end{eqnarray*}
we get $\mathrm{De}_{x_{(n-i+1)'}}^{(p-1)}\cdot v\neq 0.$ Then
   $\mathrm{De}_{x_{(n-i+1)'}}^{p-1}\cdot v$ is a highest weight vector of $L_{\overline{\frak{le}}(n)}^{\mathfrak{b}_{n+i-1}}(\lambda)$ with respect to $\mathfrak{b}_{n+i}.$

(6) Let $0\neq v$  be a highest weight vector of  $L_{\overline{\frak{le}}(n)}^{\mathfrak{b}_{n+i-1}}(\lambda)$  with respect to $\mathfrak{b}_{n+i-1}.$
 We claim $\mathrm{De}_{x_{(n-i+1)'}}^{(p-1)}\cdot v= 0.$ If not, then
$\mathrm{De}_{x_{(n-i+1)'}}^{(p-1)}\cdot v$ is a highest weight vector of $\mathfrak{b}_{n+i-1}$. Eq. (\ref{lye1}) implies $\mathrm{De}_{x_{(n-i+1)'}}^{(p-2)}\cdot v\neq 0.$ Then  $\mathrm{De}_{x_{(n-i+1)'}}^{p-2}\cdot v$ is a highest weight vector of $L_{\overline{\frak{le}}(n)}^{\mathfrak{b}_{n+i-1}}(\lambda)$  with respect to $\mathfrak{b}_{n+i}.$

\end{proof}

\section{Irreducibility of restricted Kac modules}

Recall the symplectic supergroup
$$
\mathrm{SP}(n,\mathbb{F})=\left\{A\in \mathrm{GL}(2n,\mathbb{F})\mid A^{T}JA=J\right\},
$$
where
$$
J=\left(
   \begin{array}{cc}
     0 & I_{n} \\
     -I_{n} & 0 \\
   \end{array}
 \right).
$$
The conformal symplectic supergroup $\mathrm{CSP}(n,\mathbb{F})$  is a direct product of the symplectic group $\mathrm{SP}(n, \mathbb{F})$ with the one dimensional
multiplicative supergroup $\mathbb{F}^{*}.$
Each $\phi\in \mathrm{GL}({\mathcal{O}(n,\underline{1})_{[1]},\mathbb{F}})\times \mathrm{GL}({\Lambda(n)_{[1]},\mathbb{F}})$
can be extended to a $\mathbb{Z}$-homogeneous element of $\mathrm{Aut}(\mathcal{O}(n)),$ which is still denoted  by $\phi.$ Now we define
$$f_{\phi}(x)=\phi^{-1}x\phi,\; \mbox{for}\; x\in \overline{\frak{le}}(n).$$
Let $\phi\in \mathrm{GL}({\mathcal{O}(n,\underline{1})_{[1]},\mathbb{F}})\times \mathrm{GL}({\Lambda(n)_{[1]},\mathbb{F}}).$  If $\phi\in \mathrm{CSP}(n,\mathbb{F}),$  then $f_{\phi}\in \mathrm{Aut}(\overline{\frak{le}}(n)).$
(See
\MyRoman{1} in Appendix).

\subsection{$(\mathbf{u}(\overline{\frak{le}}(n)),\mathfrak{T})$-module} Let $\mathfrak{T}$ be the canonical maximal torus of the $\mathrm{CSP}(n,\mathbb{F}).$ Then
$$\mathfrak{T}\cong\left\{\mathrm{diag}(tt_{1},\ldots,tt_{n},tt_{1}^{-1},\ldots,tt_{n}^{-1})\mid t, t_{i}\in \mathbb{F}^{*}\right\}$$
and the Lie algebra of $\mathfrak{T}$ coincides with $\bar{\mathfrak{h}}.$ Let $X(\mathfrak{T})$ be the character group of $\mathfrak{T}.$
Then
$$X(\mathfrak{T})=\sum_{i=1}^{n+1}\mathbb{Z}\Lambda_{i},$$
where, for   $t, t_{i}\in \mathbb{F}^{*},$
\[ \Lambda_{i}(\mathrm{diag}(tt_{1},\ldots,tt_{n},tt_{1}^{-1},\ldots,tt_{n}^{-1}))=\left\{
 \begin{array}{ll}
t_{i}^{-1} & if\; 1 \leq i \leq n \\
t & if\;  i=n+1.
\end {array}
\right.
\]
By definition, a rational $\mathfrak{T}$-module $V$ means that
$V=\oplus_{\lambda\in X(\mathfrak{T})}V_{\lambda},$ where
$$
V_{\lambda}=\{v\in V\mid \overline{t}(v)=\lambda(\overline{t})v  \;\mbox{for}\; \overline{t}\in \mathfrak{T}\}.
$$
Set
 $$
 \mathfrak{J}=\{\mathrm{diag}(t,\ldots, t)\mid t\in \mathbb{F}^{*}\}.
 $$
 Then we have a $\mathbb{Z}$-graded decomposition for a rational $\mathfrak{T}$-module $V=\oplus_{s\in \mathbb{Z}} V_{s}$ with
$$
V_{s}=\left\{v\in V\mid \overline{t}(v)=t^{s}v \;\mbox{for}\; \overline{t}=\mathrm{diag}(t,\ldots, t)\in \mathfrak{J}\right\}.
$$
Put
$$
\mathcal{W}_{\mathfrak{J}}(V)=\{s\in \mathbb{Z}\mid V_{s}\neq 0\}.
$$
We define the action of $\mathfrak{T}$ on $\overline{\frak{le}}(n)$  by
$$
\overline{t}(a)=f_{\overline{t}}(a) \;\mbox{for}\; \overline{t}\in \mathfrak{T}, a\in \overline{\frak{le}}(n).
$$
 $U(\overline{\frak{le}}(n))$ and its canonical subalgebras become rational $\mathfrak{T}$-modules with the action given by
$$
\mathrm{Ad}(\overline{t})(a_{1}\cdots a_{l})=\overline{t}(a_{1})\cdots \overline{t}(a_{l}),
$$
where $a_{i}\in \overline{\frak{le}}(n)$ and  $\overline{t}\in \mathfrak{T}.$
Since
$$\mathrm{Ad}(\overline{t})(x^{[p]})=\mathrm{Ad}(\overline{t})(x)^{p}  \;\mbox{for}\; x\in \overline{\frak{le}}(n), \overline{t}\in \mathfrak{T},$$
 $\mathbf{u}(\overline{\frak{le}}(n))$ is  also a rational $\mathfrak{T}$-module.
 A finite
dimensional superspace $V=V_{\bar{0}}\oplus V_{\bar{1}}$  is  called a $(\mathbf{u}(\overline{\frak{le}}(n)),\mathfrak{T})$-module if  $V$ is  a $\mathbf{u}\left(\overline{\frak{le}}(n)\right)$-module and each $V_{\bar{i}}$ ($i=0, 1$) is a   $\mathfrak{T}$-module and satisfies:

(1) the actions of $\mathfrak{h}$ coming from $\overline{\frak{le}}(n)$ and from $\mathfrak{T}$ coincide,

(2) $\overline{t}(a\cdot v)=\mathrm{Ad}(\overline{t})(a)\overline{t}(v),$ for $\overline{t}\in \mathfrak{T}, a\in \mathbf{u}(\overline{\frak{le}}(n)), v\in V.$

\begin{example}\label{lye3}
$I_{\overline{\frak{le}}(n)}(\lambda)$ and $L_{\overline{\frak{le}}(n)}^{\mathfrak{b}_{i}}(\lambda)$ ($0 \leq i\leq 2n$) are   $\left(\mathbf{u}(\overline{\frak{le}}(n)),\mathfrak{T}\right)$-modules, where $\lambda\in \Lambda_{\overline{\mathfrak{le}}(n)}$ (See
\MyRoman{1} in Appendix).
\end{example}

\subsection{Main results}

For any $\left(\mathbf{u}(\overline{\frak{le}}(n)),\mathfrak{T}\right)$-module $V,$ we define the length of $V$ as the number $|\mathcal{W}_{\mathfrak{J}}(V)|$ minus 1, denoted by $l(V)$. Write $\mathfrak{g}$ for $\frak{le}(n)$ or $\overline{\frak{le}}(n).$ Put
\begin{eqnarray*}
\Omega_{\mathfrak{g}}=\left\{\sum_{j=1}^{i-1}\varepsilon_{j}+a\varepsilon_{i}+\delta_{\mathfrak{g},\overline{\mathfrak{le}}(n)}b\delta,
\sum_{j=1}^{n}\varepsilon_{j}+\sum_{ l=i}^{n}\varepsilon_{l}+\delta_{\mathfrak{g},\overline{\mathfrak{le}}(n)}b\delta\mid a, b\in \mathbb{F}_{p}, i \leq n\right\}.
\end{eqnarray*}
A weight $\lambda\in \Lambda_{\mathfrak{g}}$ is said to be atypical if $\lambda\in \Omega_{\mathfrak{g}};$ Otherwise,  $\lambda$ is said to be typical.

 For $\lambda\in \Lambda_{\overline{\mathfrak{le}}(n)} $ and $i$ ($1\leq i\leq 2n$),  there is uniquely a weight in $\Lambda_{\overline{\mathfrak{le}}(n)}$, which is denoted by $\lambda_{i}$,    such that $ L_{\overline{\frak{le}}(n)}^{\mathfrak{b}_{i}}(\lambda_{i})=L_{\overline{\frak{le}}(n)}^{\mathfrak{b}_{0}}(\lambda).$

\begin{theorem}\label{tt1}
 For $\lambda\in \Lambda_{\overline{\mathfrak{le}}(n)}$, the restricted Kac module $I_{\overline{\frak{le}}(n)}(\lambda)$ is irreducible if and only if  $\lambda$ is typical.
\end{theorem}
\begin{proof}
Let $\upsilon_{0}$ and $\upsilon_{2n}$  be the highest weight vectors of  $L_{\overline{\frak{le}}(n)}^{\mathfrak{b}_{0}}(\lambda)$ with respect to $\mathfrak{b}_{0}$ and $\mathfrak{b}_{2n}$, respectively (see Subsection 2.3), and suppose
$$\overline{t}(\upsilon_{0})=t^{l_{0}}\upsilon_{0},\;\overline{t}(\upsilon_{2n})=t^{l_{2n}}\upsilon_{2n} \;\mbox{for}\; \overline{t}=\mathrm{diag}(t,\ldots, t)\in \mathfrak{J}.$$
We have $l\left(L_{\overline{\frak{le}}(n)}^{\mathfrak{b}_{0}}(\lambda)\right)=l_{2n}-l_{0}.$
Then $ l\left(L_{\overline{\frak{le}}(n)}^{\mathfrak{b}_{0}}(\lambda)\right)\leq l\left(I_{\overline{\frak{le}}(n)}(\lambda)\right)=pn$ and $I_{\overline{\frak{le}}(n)}(\lambda)$ is irreducible if and only if $l\left(L_{\overline{\frak{le}}(n)}^{\mathfrak{b}_{0}}(\lambda)\right)= pn.$ Then it suffices  to show that  $l\left(L_{\overline{\frak{le}}(n)}^{\mathfrak{b}_{0}}(\lambda)\right)=pn$ if and only if  $\lambda$ is typical. By Proposition \ref{pl1}, we have the following facts:

 (a) For $2\leq k\leq n,$   $\lambda_{k-1}(\mathfrak{h}_{k})=0$ if and only if
$$ \lambda=b\delta \;\mbox{or} \;\sum_{j=1}^{k-1}\varepsilon_{j}+a\varepsilon_{k}+b\delta, \;  0\not=a, b\in \mathbb{F}_{p}.$$

 (b) For $k= n,$    $\lambda_{k}(\mathfrak{h}_{2n-k})=0$ if and only if
$$\lambda= b\delta\; \mbox{or}\;\sum_{j=1}^{n-1}\varepsilon_{j}+a\varepsilon_{n}+b\delta, 0\not=a, b\in \mathbb{F}_{p}.$$

  (c) For $k>n,$  $\lambda_{k}(\mathfrak{h}_{2n-k})=0$ if and only if
$$\lambda= b\delta\;\mbox{or}\;\sum_{j=1}^{2n-k}\varepsilon_{j}+a\varepsilon_{2n-k}+2\sum_{ j=2n-k+1}^{n}\varepsilon_{j}+b\delta, 0\not=a, b\in \mathbb{F}_{p}.$$

  (d) If $\lambda=\sum_{j=1}^{n-1}\varepsilon_{j}+a\varepsilon_{n}+b\delta$, where $0\neq a, b\in \mathbb{F}_{p}$, then $\lambda_{n}=a\varepsilon_{n}+(b-n+1)\delta.$

  (e) For $k>n$ and  $\lambda=\sum_{j=1}^{2n-k}\varepsilon_{j}+a\varepsilon_{2n-k}+2\sum_{ j=2n-k+1}^{n}\varepsilon_{j}+b\delta, 0\not=a, b\in \mathbb{F}_{p}$,   $ \lambda_{k}=a\varepsilon_{2n-k}+(b-2n+k)\delta.$

By  the facts (a)--(e) and Proposition \ref{pl1},  we have
\[ l\left(L_{\overline{\frak{le}}(n)}^{\mathfrak{b}_{0}}(\lambda)\right)=\left\{
 \begin{array}{ll}
0, & if \; \lambda=b\delta, b\in \mathbb{F}_{p}\\
pn-1,& if \; \lambda=\sum_{j=1}^{i-1}\varepsilon_{j}+a\varepsilon_{i}+b\delta, 1\leq i\leq n-1, 0\not=a, b\in \mathbb{F}_{p}\\
pn-p,& if \; \lambda=\sum_{j=1}^{n-1}\varepsilon_{j}+a\varepsilon_{n}+b\delta, 0\not=a, b\in \mathbb{F}_{p}\\
pn-1,& if \;\lambda=\sum_{j=1}^{n}\varepsilon_{j}+\sum_{ j=i}^{n}\varepsilon_{j}+b\delta, b\in \mathbb{F}_{p}, i < n\\
pn,& \mbox{otherwise}.
\end {array}
\right.
\]
That is, $l\left(L_{\overline{\frak{le}}(n)}^{\mathfrak{b}_{0}}(\lambda)\right)=pn$ if and only if  $\lambda$ is typical.  The proof is complete.
\end{proof}
As a corollary of  Theorem \ref{tt1}, we have
\begin{theorem}\label{c1}
For $\lambda\in \Lambda_{\mathfrak{le}(n)},$  $I_{\frak{le}(n)}(\lambda)$ is irreducible if and only if  $\lambda$ is typical.
\end{theorem}

\addcontentsline{toc}{chapter}{Appendix}
\appendix
\renewcommand\thesection{Appendix}
\section{}
\renewcommand\thesection{\arabic{section}}
\renewcommand{\thelemma}{A.\arabic{lemma}}
\renewcommand{\theequation}{A.\arabic{equation}}

\MyRoman{1}. Let
\[\phi(x_{i})=\left\{
 \begin{array}{ll}
\sum_{j=1}^{n}a_{ji}x_{j},&  \mbox{if} \; i\leq n \\
\sum_{j=n+1}^{2n}a_{ji}x_{j},&   \mbox{if}\; i>n.
\end {array}
\right.
\]
If $\phi=\mathrm{diag}(t,\ldots, t)\in\mathfrak{J},$
then
\begin{eqnarray}\label{e1}
f_{\phi}\left(\sum_{i=1}^{2n}x_{i}\partial_{i}\right)&=&\sum_{i=1}^{2n}\phi(x_{i})f_{\phi}(\partial_{i})=\sum_{i=1}^{2n}tx_{i}t^{-1}\partial_{i}=
\sum_{i=1}^{2n}x_{i}\partial_{i}
\end{eqnarray}
and
\begin{eqnarray}\label{e2}
f_{\phi}(\mathrm{De}_{f})=t^{\mathrm{deg}f-2}\mathrm{De}_{f} \;\mbox{for}\; f\in \mathcal{O}(n).
\end{eqnarray}
If $\phi\in \mathrm{SP}(n, \mathbb{F}),$ then for   $1 \leq i, j\leq n,$  we have
$\sum_{k=1}^{n}a_{ik}a_{j'k'}= \delta_{ij}.$
Then
\begin{eqnarray}\label{e3}\nonumber
f_{\phi}\left(\sum_{i=1}^{2n}x_{i}\partial_{i}\right)&=&\sum_{i=1}^{2n}\phi(x_{i})f_{\phi}(\partial_{i})\\\nonumber
&=&\sum_{i=1}^{n}\sum_{j=1}^{n}a_{ji}x_{j}\sum_{k=1}^{n}a_{k'i'}\partial_{k}+\sum_{i=n+1}^{2n}\sum_{j=n+1}^{2n}a_{ji}x_{j}\sum_{k=n+1}^{2n}
a_{k'i'}\partial_{k}\\\nonumber
&=&\sum_{j,k=1}^{n}\left(\sum_{i=1}^{n}a_{ji}a_{k'i'}\right)x_{j}\partial_{k}+\sum_{j,k=n+1}^{2n}\left(\sum_{i=n+1}^{2n}a_{ji}a_{k'i'}\right)x_{j}\partial_{k}\\
&=&\sum_{j,k=1}^{2n}\delta_{jk}x_{j}\partial_{k}=\sum_{i=1}^{2n}x_{i}\partial_{i}
\end{eqnarray}
and
\begin{eqnarray}\label{e4}
f_{\phi}(\mathrm{De}_{f})=\mathrm{De}_{\phi(f)} \;\mbox{for}\; f\in \mathcal{O}(n).
\end{eqnarray}
Eqs. (\ref{e2}) and (\ref{e4}) imply that
$$
f_{\phi}([\mathrm{De}_{f}, \mathrm{De}_{g}])=[f_{\phi}(\mathrm{De}_{f}),f_{\phi}(\mathrm{De}_{g})] \;\mbox{for}\; f, g\in \mathcal{O}(n).
$$
Using Eqs.  (\ref{e1}--\ref{e4}) and the fact that $\phi$ is a $\mathbb{Z}$-homogeneous automorphism of $\mathcal{O}(n)$, we have
$$
f_{\phi}\left(\left[\sum_{i=1}^{n}x_{i}\partial_{i}, \mathrm{De}_{f}\right]\right)=\left[f_{\phi}\left(\sum_{i=1}^{n}x_{i}\partial_{i}\right),f_{\phi}(\mathrm{De}_{f})\right] \;\mbox{for}\; f\in \mathcal{O}(n).
$$
Therefore, $f_{\phi}\in \mathrm{Aut}\left(\overline{\frak{le}}(n)\right).$\\

\noindent\MyRoman{2}.  As in \cite{SY}, we have
$$X(\mathfrak{T})/ p X(\mathfrak{T})\cong \Lambda_{\overline{\frak{le}}(n)}.$$
Then for $a\in \mathbf{u}(\overline{\frak{le}}(n))$ and a highest weight vector $\upsilon_{\lambda}$ of  $L_{\overline{\frak{le}}(n)}^{\mathfrak{b}_{0}}(\lambda)$ with respect to $\mathfrak{b}_{0}$, we can define
$$
\overline{t}(a\otimes v_{\lambda})=\mathrm{Ad}(\overline{t})(a)\otimes  \lambda(\overline{t})(v_{\lambda}).
$$
Clearly, $I_{\overline{\frak{le}}(n)}(\lambda)_{\bar{i}}$, $i=0, 1$, is a   $\mathfrak{T}$-module and
$$\overline{t}(a\cdot v)=\mathrm{Ad}(\overline{t})(a)\overline{t}(v) \;\mbox{for}\;\overline{t}\in \mathfrak{T}, a\in \mathbf{u}(\overline{\frak{le}}(n)), v\in I_{\overline{\frak{le}}(n)}(\lambda).$$
We claim the action of $\mathfrak{T}$ on $\overline{\frak{le}}(n)$ coincides with that of $\bar{\mathfrak{h}}.$ For $$\overline{t}=\mathrm{diag}(tt_{1},\ldots,tt_{n},tt_{1}^{-1},\ldots,tt_{n}^{-1})\in \mathfrak{T},\;x^{(\underline{r})}\in \mathcal{O}(n)\;\mbox{and}\;h\in \bar{\mathfrak{h}},$$ we can check the following equations:
\begin{eqnarray*}
\overline{t}(\mathrm{De}_{x^{(\underline{r})}})&=&t^{\mathrm{deg}x^{(\underline{r})}-2}t_{1}^{r_{1}}\cdots t_{n}^{r_{n}}t_{1}^{-r_{n+1}}\cdots t_{n}^{-r_{2n}}\mathrm{De}_{x^{(\underline{r})}}\\
&=&\left(\sum_{i=1}^{n}(r_{i'}-r_{i})\Lambda_{i}+(\mathrm{deg}x^{(\underline{r})}-2)\Lambda_{n+1}\right)(\overline{t})\mathrm{De}_{x^{(\underline{r})}},
\end{eqnarray*}
\begin{eqnarray*}
[h,\mathrm{De}_{x^{(\underline{r})}}]=\left(\sum_{i=1}^{n}(r_{i'}-r_{i})\varepsilon_{i}+(\mathrm{deg}x^{(\underline{r})}-2)\delta\right)(h)\mathrm{De}_{x^{(\underline{r})}},
\end{eqnarray*}
\begin{eqnarray*}
\overline{t}\left(\sum_{i=1}^{2n}x_{i}\partial_{i}\right)=\sum_{i=1}^{2n}x_{i}\partial_{i}=0(\overline{t})\sum_{i=1}^{2n}x_{i}\partial_{i},
\end{eqnarray*}
\begin{eqnarray*}
\left[h,\sum_{i=1}^{2n}x_{i}\partial_{i}\right]=0=0(h)\sum_{i=1}^{2n}x_{i}\partial_{i}.
\end{eqnarray*}
Summarizing, the action of $\mathfrak{T}$ on $\overline{\frak{le}}(n)$ coincides with that on $\bar{\mathfrak{h}}.$ Then $I_{\overline{\frak{le}}(n)}(\lambda)$ is a $\left(\mathbf{u}(\overline{\frak{le}}(n)),\mathfrak{T}\right)$-module. Similarly, $L_{\overline{\frak{le}}(n)}^{\mathfrak{b}_{i}}(\lambda)$ is also  $\left(\mathbf{u}(\overline{\frak{le}}(n)),\mathfrak{T}\right)$-module, $ 0 \leq i\leq 2n$.
\\

\noindent \textbf{Acknowledgements.} The authors  are grateful to Professor Chaowen Zhang for several conversations  and suggestions on this topic.

\end{document}